\documentclass{amsart}
\usepackage[english]{babel}

\newcommand{\Ps}{{\mathbf{P}}}
\newcommand{\Z}{{\mathbf{Z}}}
\newcommand{\C}{{\mathbf{C}}}

\newcommand{\Q}{{\mathbf{Q}}}
\newcommand{\A}{{\mathbf{A}}}
\newcommand{\F}{{\mathbf{F}}}
\newcommand{\oQ}{\overline{\mathbf{Q}}}

 \newcommand{\p}{\mathfrak{p}}

\renewcommand{\phi}{\varphi}
    \newtheorem{lemma}{Lemma}[section]
    \newtheorem{proposition}[lemma]{Proposition}
    \newtheorem{theorem}[lemma]{Theorem}

   \theoremstyle{definition}
    
    \newtheorem{notation}[lemma]{Notation}
    
    \newtheorem{remark}[lemma]{Remark}
    
    \DeclareMathOperator{\rank}{rank}

\DeclareMathOperator{\Frob}{{Frob}}
\DeclareMathOperator{\Tr}{{Tr}}
\DeclareMathOperator{\et}{{\acute{e}t}}

\DeclareMathOperator{\sing}{{sing}}

\DeclareMathOperator{\prim}{{prim}}

\DeclareMathOperator{\cha}{{char}}

\DeclareMathOperator{\coker}{{coker}}

\DeclareMathOperator{\Gr}{{Gr}}
\DeclareMathOperator{\rig}{{rig}}

\usepackage[all]{xy}
\begin{document}
\title[Elliptic threefold]{The elliptic threefold $y^2=x^3+16s^6+16t^6-32(t^3s^3+t^3+s^3)+16$}
\author{Remke Kloosterman}
\address{Institut f\"ur Al\-ge\-bra\-ische Geo\-me\-trie, Leibniz Universit\"at
Hannover, Wel\-fen\-gar\-ten 1, 30167 Hannover, Germany}
\email{kloosterman@math.uni-hannover.de}

\begin{abstract} We present a method to calculate the rank of $E(\oQ(s,t))$ for the elliptic curve mentioned in  the title. This method uses a generalization of a method from Van Geemen and Werner to calculate $h^4(Y)$ for nodal hypersurfaces $Y$.
\end{abstract}
\subjclass{}
\keywords{}
\date{\today}
\thanks{The author wishes to thank Klaus Hulek, Matthias Sch\"utt and Jaap Top for comments on an earlier version of this paper.}
\maketitle

\section{Introduction}\label{secInt}
In this paper we study the elliptic threefold associated with 
\begin{equation} \label{mainEqn} y^2=x^3+16s^6+16t^6-32(t^3s^3+t^3+s^3)+16.\end{equation}
This equation defines a singular threefold $T$ in $\A^4$, which is birational to an elliptic 3-fold $X$. Since an elliptic 3-fold $ X $ is smooth and projective (by definition) we have to apply several transformations in order to obtain $X$. For more on this see \cite[Section 4]{ell3HK}.

At the same time, we can consider (\ref{mainEqn}) as the equation for an elliptic curve $E/\Q(s,t)$. The groups $E(\oQ(s,t))$ and $E(\Q(s,t))$ are isomorphic to the group of rational sections of the elliptic fibration on $X$ and the subgroup of rational sections that can be defined over $\Q$, respectively.

One can find several points on $E(\oQ(s,t))$: For example,  $P_1:=(-4s,4(t^3-s^3-1))$  is a point on $E(\oQ(s,t))$. Using the symmetry in $s$ and $t$ of  (\ref{mainEqn}) we find additional points $P_2:=(-4t,4(s^3-t^3-1))$ and
$P_3:=(-4ts,4(1-s^3-t^3))$. Finally, using the complex multiplication on $E$ we find three additional independent points $\omega P_1,\omega P_2, \omega P_3$, where $\omega$ is the complex multiplication of the elliptic curve $E$, i.e., we multiply the $x$-coordinate with a fixed third root of unity.

We listed so far  six points  and they turn out to be independent, i.e., they define a rank 6 subgroup of $E(\oQ(s,t))$. Using the methods of Hulek and the author \cite{ell3HK} we concluded in \cite{elljcst} that $\rank E(\oQ(s,t))=6$. So we have found explicit generators for a finite index subgroup of $E(\oQ(s,t))$ and we could stop at this point. However, we would like to  provide a different proof for the fact that $\rank E(\oQ(s,t))=6$.

As in the papers \cite{ell3HK,elljcst} we start by using that $\rank E(\oQ(s,t))+1$ equals the rank of  $ H^4(Y,\Z) \cap H^{2,2}(H^4(Y,\C))$ for a certain threefold $Y$ in the weighted projective space $ \Ps(2,3,1,1,1)$. In our case it turns out that $\Sigma:=Y_{\sing}$ consists of 9 singularities, all of which are of type $D_4$. Such a singularity is isolated and weighted homogeneous. A method due to Dimca \cite{DimBet} yields that the 4th cohomology group with support in  $\Sigma$, $H^4_\Sigma(Y)$, is  pure of type  $(2,2)$ and the same holds for $H^4(Y)$. Hence $\rank E(\oQ(s,t))=h^4(Y)-1$.

We provide a different method to calculate $h^4(Y)$ than in \cite{ell3HK,elljcst}. In \cite{ell3HK} we used
\[ H^4(Y)_{\prim} \cong \coker \psi: H^4(\Ps(2,3,1,1,1)\setminus Y) \to \oplus_{p\in \Sigma} H^4_p(Y)\]
and we indicated a method to explicitly calculate the map $\psi$. This means that we can represent $\psi$ (with respect to some choice of basis) as a matrix $M$ and then calculate $\rank M$. In the present example this is still feasible and we obtain a $21 \times 9$ matrix. Using a computer one easily calculates the rank of a matrix of this size.
However, if the number of singularities increases, or if the degree of the equation increases, then the size of the matrix also increases and this leads rapidly to matrices that are quite big. 

In \cite{vGW} Van Geemen and Werner provided a different approach for computing $h^4(Y)$. For a nodal hypersurfaces $X$ in  $\Ps^4$, they first related the Betti numbers of $X$ with the Betti numbers of a fixed resolution of singularities $\tilde{X}$. Since the Euler characteristic of $X$ and $\tilde{X}$ only depends on the degree of $X$ and the number of double points, and since the $h^i(\tilde{X})$ are known for $i\neq 2,3,4$ we get a linear relation between $h^2$, $h^3$ and $h^4$.  Poincar\'e duality yields $h^2=h^4$. Hence to determine all the Betti numbers it suffices to determine $h^4$.

Let $E$ be the exceptional divisor of $\tilde{X}\to X$ and let $p$ be a prime of good reduction of $\tilde{X}$  such that $\Frob_p$ acts as multiplication by $p$ on $H^2_{\et}(E,\Q_\ell)$ (i.e., all singularities of $Y$ are defined over $\F_p$ and at each singularity the tangent cone is isomorphic to $\Ps^1\times \Ps^1$ rather than a twist of $\Ps^1\times\Ps^1$). The Lefschetz trace formula yields
\[ \mid \#\tilde{X}(\F_p)-p^3-(p^2+p)(h^4(\tilde{X})) -p-1 \mid \leq p^{3/2}h^3(\tilde{X}) = p^3(\chi(\tilde{X})-2-2h^4(\tilde{X}) )\]
Now $\#\tilde{X}(\F_p)$ can  be determined by a naive point count with a computer. All other quantities, except for $h^4$, are known. 
This inequality yields a lower and an upper bound for $h^4$, and if $p$ is sufficiently large these bounds differ by less than 1, hence $h^4$ is the unique integer in that particular interval.

It is easy to check that the arguments of Van Geemen and Werner can  be extended to hypersurfaces $X$ in weighted projective 4-space, admitting a resolution of singularities $\tilde{X}$ such that the exceptional divisor $E$ is a union of rational surfaces. This includes all hypersurfaces with at most isolated $ADE$-singularities.

We slightly extend this method. The main difference with Van Geemen and Werner is that we work with the rigid cohomology of $Y$ rather than the \'etale cohomology of $\tilde{Y}$.
In an upcoming paper we prove that our method can be extended to hypersurfaces $(Y,p)$ such that $H^4_p(Y)$ has a Hodge structure of Tate type.
 One can easily provide examples where this holds, but where Van Geemen-Werner does not work. 
It should be said that the method of Hulek and the author works for an even bigger class of varieties: e.g., it works  for hypersurfaces such that  the cokernel of $H^4(\Ps \setminus Y)\to \oplus H^4_p(Y)$ is a Hodge structure of Tate type. From a theoretical point of view, it should be remarked that the method presented in this paper and the method of Van Geemen and Werner work only over number fields, whereas the method of Hulek and the author works for any variety defined over the complex numbers.

\section{The hypersurface $Y$}
Let $K$ be either a subfield of the complex numbers or a finite field not of characteristic 2 or 3, such that $K$ contains a primitive third root of unity $\omega$.
Consider the equation (\ref{mainEqn}). The hypersurface in $\A^4$ defined by this equation is birationally equivalent to the hypersurface $Y\subset \Ps(2,3,1,1,1)$ defined by
\[ y^2=x^3+16(z_0^6+z_1^6+z_2^6-2(z_0^3z_1^3+z_1^3z_2^3+z_0^3z_2^3)). \]
Let $C\subset \Ps^3$ be the curve defined by
\[ z_0^6+z_1^6+z_2^6-2(z_0^3z_1^3+z_1^3z_2^3+z_0^3z_2^3)=0.\]
\begin{lemma} Suppose $K$ is algebraically closed and not of characteristic $2,3$. 
Let $\omega$ be a primitive third root of unity. Then the hypersurface $Y_{\sing}$ consists of the $9$ points 
\[ (0:0:\omega^i:1:0), (0:0:\omega^j:0:1)\mbox{ and }(0:0:0:\omega^k:1)\]
for $i,j,k\in \{0,1,2\}$, and at each of these points $Y$ has a $D_4$-singularity.
\end{lemma}

\begin{proof}
The singular locus of $\Ps(2,3,1,1,1)$ consists of $(1:0:0:0:0)$ and $(0:1:0:0:0)$. Both points do not lie on $Y$, hence to determine $Y_{\sing}$ it suffices to consider the partials of the defining polynomial. It turns out these partials vanish at 9 points, namely $(0:0:\omega^i:1:0)$, $(0:0:\omega^i:0:1)$ and $(0:0:0:\omega^i:1)$. 

Actually these 9 singularities are dual to the 9 cusps of $C$. Since a cusp has local equation $t^2+s^3$ it turns out that $Y$ has local equation $-y^2+x^3+t^2+s^3$ at each singular point, i.e., $Y$ has 9 $D_4$-singularities. 
\end{proof}

\begin{remark} 
Since the curve $C$ is a sextic with 9 cusps, it is  the dual of a smooth cubic.
The 9 singularities correspond with the 9 flexes on the smooth cubic. The  sextic $C$ is the dual curve of $z_0^3+z_1^3+z_2^3=0$, i.e., the elliptic curve with $j$-invariant $0$.
\end{remark}

\begin{notation} For a prime power $q=p^r$ denote with $\Q_q$ the unique degree $r$ unramified extension of $\Q_p$. Equivalently, $\Q_q$ is the field of fractions of the Witt vectors of $\F_q$.

For a closed subset $\Sigma\subset Y$ we denote with $H^i_\Sigma(Y)$ the $i$-th cohomology group with support in $\Sigma$. Here $H^i_\Sigma$ denotes singular cohomology in case $K\subset \C$ or rigid cohomology with $\Q_q$-coefficients  in case $K=\F_q$.
\end{notation}

\begin{lemma} Let $P=(0,0,0,\omega^i)\in Y_{\sing}$. Let $S\subset \Ps(2,3,2,3)$ be the surface $-y^2+x^3-64s^3+144\omega^it^ 2$. Then $H^4_P(Y)\cong H^2(S)_{\prim}(-1)$.
\end{lemma}

\begin{proof}
Setting $z_1=1$ and moving $(0,0,0,\omega^i)$ to $(0,0,0,0)$ yields a local defining polynomial for $Y$:
\[ -y^2+x^3-64s^3+144 \omega^i t^2 + \mbox{ higher order terms}\]
Suppose for the moment that $K\subset \C$.
Since this singularity is semi-weighted homogeneous it follows from Dimca's work that the local cohomology $H^4_P(Y)$ is isomorphic to $H^2(S)_{\prim}(-1)$ where $S$ is  the (quasi-smooth) surface $y^2=x^3-64s^3+144 \omega^i t^2$ in $\Ps(2,3,2,3)$. 

If $K$ is a field of positive  characteristic then we can obtain similar results. In a neighborhood of $P$ we have that $(Y,P)$ is given by $-y^2+x^3-64s^3+144t^2=0$. Hence $Y\setminus \{P\}$ is locally a $K^*$-bundle over $S\subset \Ps(2,3,2,3)$. As in  characteristic zero it follows then from the Leray spectral sequence   that $H^4_{p,\rig}(Y)\cong H^2(S)_{\prim,\rig}(-1)$. 
\end{proof}

Setting $s_1=4s, t_1=12t$ one finds that $S$ is isomorphic to the surface $-y^2+x^3-s_1^3+t_1^2=0$.

\begin{remark}
The group $H^2(S)_{\prim}$ can be calculated as follows: first note that the surface $S$ is smooth. Hence, in case $K=\F_p$ it follows from the work of Baldassarri-Chiarellotto \cite{BalChi} that $H^2(S,\Q_q)$ equals the second algebraic de Rham cohomology group (with $\Q_q$-coefficients) of a lift of $S$ to characteristic zero.

This reduces the problem to the case $K=\C$. By the work of Griffiths and Steenbrink \cite{SteQua} we know that the second primitive cohomology is isomorphic with degree 2 part of the Jacobian ring of the defining equation (with  $\deg(s_1)=\deg(x)=2$ and $\deg(y)=\deg(t_)=3$). The Jacobian ring is isomorphic to $R[x,y,s_1,t_1]/(x^2,y,s_1^2,t_1)$, and the degree 2 part is generated by $x$ and $s$  hence is two-dimensional. Hence $h^0(S)=h^4(S)=1$, $h^2(S)=3$.
\end{remark}

\begin{lemma} 
Suppose $K$ is a finite field with $q$ elements (and $q\equiv 1 \bmod 3$). Then $\Frob_q$ acts as multiplication by $q$ on the  rigid cohomology group $H^2(S)$.
\end{lemma}

\begin{proof} 
Since $t_1^2-y^2$ factors over $\F_q$ in two distinct factors we have that $S$ is isomorphic to the surface
\[ t_1y+x^3-s_1^3=0.\]
This implies that for each choice of $x$ and $s_1$ such that $x^3-s_1^3\neq 0$, we have $q+1$ possibilities for $(y,t_1)$ and if $x^3-s_1^3=0$ we have $2q-1$ possibilities for $(y,t_1)$.

Since $q\equiv 1 \bmod 3$ there are three points in  $\Ps^1$ such that $x^3-s_1^3$ vanishes, hence $S$ has $(q-2)(q-1)+3(2q-1)=q^2+3q-1$ points with at least one of $x,s$ nonzero. If $x=s_1=0$ then we have the equation $t_1y=0$ yielding the points $(1:0:0:0)$ and $(0:0:0:1)$. In total we obtain $q^2+3q+1$ points.
From the Lefschetz trace formula it follows now that Frobenius has trace $3q$ on $H^2_{\rig}(Y)$. Since each eigenvalue of Frobenius has complex absolute value $q$, and $h^2(S)=3$, it follows that Frobenius acts as multiplication by $q$ on $H^2_{\rig}(S)$.
\end{proof}

From $H^4_q(Y)\cong H^2(S)(-1)$ it follows that $\Frob_q^*$ acts as multiplication by $q^2$ on $H^4_q(Y)$.

We summarize the results of this section:
\begin{proposition} \label{prpSec2} Let $\Sigma$ be the singular locus of $Y$. Then $\Sigma$ consists of 9 points, and at each point $Y$ has a $D_4$ singularity. 
Moreover, if $K\subset \C$ then
$H^4_{\Sigma}(Y)=\Q(-2)^{18}$ and if $K=\F_q$  and $q\equiv 1 \bmod 3$ then  $H^4_{\Sigma}(Y)=\Q_q(-2)^{18}$.
\end{proposition}

\section{Calculating $h^i(Y)$ and $\rank E(\Q(s,t))$}
In this section $H^\bullet$ refers to the de Rham cohomology of $Y_{\C}$ if  $K$ is a number field, and  to rigid cohomology if $K$ is a finite field. All exact sequence are either  exact sequences of mixed Hodge structures  or exact sequence of $\Q_q$-vector spaces with Frobenius action. The cohomology groups $H^i(Y)$ come with a weight filtration, we use $\Gr^W_k H^i(Y)$ to indicate the graded pieces of this filtration.

Suppose for the moment that $K$ is a number field, and $\p$ is a prime not lying over
 2 or 3. Let  $\tilde{Y}$ be a minimal resolution of singularities of $Y$ (i.e., blowing up the $D_4$-singularity once, and then blowing up the resulting three $A_1$-singularities.) An easy calculation shows that $\tilde{Y}$ has good reduction modulo $\p$, hence $\tilde{Y}_{\F_\p}$ is a resolution of singularities of $Y_{\F_\p}$. 

Assume again that $K$ is either a number field or a finite field not of characteristic 2 and 3, and $K$ contains a primitive root of unity. The above discussion  shows that if $K$ is not of characteristics 2 or 3 then $Y$ has a resolution of singularities.

Let $E$ be the exceptional divisor of $\tilde{Y} \to Y$.
In the following, we denote $\Sigma=Y_{\sing}$  and $Y^*=Y\setminus \Sigma=\tilde{Y}\setminus E$.

 \begin{lemma} \label{lemsup} $H^i_E(\tilde{Y})=0$ for $i\neq 2,4,6$.

 \end{lemma}
 \begin{proof}

  Since $\tilde{Y}$ is smooth, a standard argument involving Poincar\'e duality and the Gysin exact sequence yields that $H^i_E(\cong{Y})\cong H^{6-i}(E)^*(-3)$. The resolution of a $D_4$-singularity is well-known. The exceptional divisor is a union of rational surfaces. From  this it follows easily  that $H^i(E)=0$ for $i\neq 0,2,4$.
 \end{proof}

\begin{lemma} The exact sequence of the pair $(Y,Y^*)$ induces the following exact sequence
 \begin{equation} \label{w4eqn} 0\to \Gr^W_4 H^3(Y^*) \to H^4_\Sigma(Y) \to H^4(Y)_{\prim} \to 0. \end{equation}
\end{lemma}

\begin{proof}
 Since $Y$ has isolated singularities one has that $h^i_\Sigma=0$ for $i\neq3,4,6$. Hence a part of the sequence of the pair $(Y,Y^*)$ reads
 \begin{eqnarray*}  0\to H^2(Y) & \to & H^2(Y^*)\to H^3_\Sigma(Y)\to H^3(Y) \to H^3(Y^*) \to\\&\to& H^4_\Sigma(Y) \to H^4(Y) \to H^4(Y^*)\to 0. \end{eqnarray*}
From the Lefschetz hyperplane theorem \cite[Theorem B22]{Dim} it follows that $H^2(Y)=\Q(-1)$. 
  A standard argument as in \cite{DimBet} shows that $H^4(Y^*)$ is one-dimensional. From this it follows that
 \[  0\to  H^2(Y^*)_{\prim} \to H^3_\Sigma(Y)\to H^3(Y) \to H^3(Y^*)\to H^4_\Sigma(Y) \to H^4(Y)_{\prim} \to 0\]
 is exact.
 From Proposition~\ref{prpSec2} it follows that $H^4_\Sigma(Y)$ is pure of  weight 4. Since $Y$ is compact, the group $H^3(Y)$ has no elements of weight $\geq 4$. Hence we obtain the following exact sequence:
\[ 0\to \Gr^W_4 H^3(Y^*) \to H^4_\Sigma(Y) \to H^4(Y)_{\prim} \to 0. \]
 \end{proof}

Now $H^1(Y^*)$ vanishes since it is an extension of $H^1(Y)$ and $H^2_\Sigma(Y)$ and both these groups vanish. From Lemma~\ref{lemsup} it follows that also $H^3_E(\tilde{Y})$ and $H^5_E(\tilde{Y})$ vanish. Hence   the sequence for the pair $(\tilde{Y},Y^*)$ reads
 \begin{eqnarray}\label{h2seq} 0\to H^2_E(\tilde{Y}) \to H^2(\tilde{Y}) \to H^2(Y^*) \to  0  \end{eqnarray}
 and
 \[ 0\to H^3(\tilde{Y}) \to H^3(Y^*) \to H^4_E(\tilde{Y}) \to H^4(Y) \to H^4(Y^*) \to 0.\]
 
 Moreover one has the so-called \emph{Mayer-Vietoris exact sequence for the discriminant square} \cite[Corollary 5.37]{PSbook}. Part of this sequence is the following
 \begin{eqnarray} \label{MVseq} 0 =H^3(E) \to H^4(Y) \to H^4(\tilde{Y}) \to H^4(E) \to 0. \end{eqnarray}
 
\begin{proposition}\label{prpw23}
There is an exact sequence
\[ 0 \to \Gr^W_2 H^3(Y)^*(-3) \to H^4_\Sigma (Y)  \to H^4(Y)_{\prim} \to 0 \]
\end{proposition}

\begin{proof}
It suffices to prove $\Gr^W_2 H^3(Y)^*(-3)\cong \Gr^W_4H^3(Y^*)$. Then the proposition follows from (\ref{w4eqn}).

Substituting $H^2_E(\tilde{Y})=H^4(E)^*(-3)$ and $H^2(Y)=H^4(Y)^*(-3)$ in (\ref{h2seq}) we obtain
\[ 0 \to H^4(E)^*(-3) \to H^4(\tilde{Y})^*(-3)\to H^2(Y^*)  \to   0. \]
Combining this with the dualized version of (\ref{MVseq}) it follows that $H^2(Y^*)\cong H^4(Y)^*(-3)$
and $H^2(Y^*)_{\prim}\cong H^4(Y)_{\prim}^*(-3)$. This yields an exact sequence
\[ 0 \to H^4(Y)_{\prim}^*(-3)\to H^3_\Sigma(Y) \to H^3(Y) \to H^3(Y^*) \to H^4_\Sigma(Y) \to H^4(Y)_{\prim} \to 0.
 \]

Since $Y^*$ is smooth, one has that $\Gr^W_i H^3(Y^*)=0$ for $i\leq 2$. Taking $\Gr_2^W$ in the above exact sequence yields that
\[ 0 \to \Gr^W_4 H^4(Y)_{\prim}^*(-3) \to \Gr^W_2 H^3_\Sigma(Y) \to \Gr^W_2 H^3(Y)\to 0\]
is exact. Since $\Gr^W_2 H^3_\Sigma(Y) \cong \Gr^W_4 H^4_\Sigma(Y)^*(-3) \cong H^4_\Sigma(Y)(1)$, we obtain (after dualizing and twisting) that
\[ 0 \to \Gr^W_2H^3(Y)^*(-3) \to H^4_\Sigma(Y) \to H^4_{prim}(Y) \to 0\]
is exact.
\end{proof}
 
\begin{lemma}\label{lemchi} We have $\chi(Y)=-2$.

\end{lemma}
\begin{proof} Suppose first that $\cha(K)=0$.  A smooth degree 6 hypersurface $Y'$ in $\Ps(2,3,1,1,1)$ has $h^3(Y')=42$ (this can be obtained, e.g.,, by using the results from \cite[Section 4]{SteQua}) and therefore $\chi(Y')=-38$. Since $Y$ has $9$ $D_4$ singularities (which have Milnor number 4) we have by, e.g.,  \cite[Corollary 5.4.4]{Dim} that $\chi(Y)=-38+36=-2$. 

Now, we have that $\chi(\tilde{Y })=\chi(\tilde{Y }_{\F_p})$, since $p$ is a prime of good reduction. For the same reason also $\chi(E)=\chi(E_{\F_p})$ and $\chi(\Sigma)=\chi(\Sigma_{\F_p})$ holds.

From the Mayer-Vietoris sequence for the discriminant square \cite[Corollary 5.37]{PSbook} we obtain that $\chi(Y)=\chi(\tilde{Y})+\chi(\Sigma)-\chi(E)$ and the same relation in positive characteristic, hence $\chi(Y)=\chi(Y_{\F_p})$.
%
\end{proof}

Denote with $w_{3,k}=\dim \Gr^W_k H^3(Y)$.
Since $E$ is a union of rational surfaces, not intersecting in curves, it follows that $H^2(E)$ is pure of weight 2. Since $\tilde{Y}$ is smooth it follows that $H^3(\tilde{Y})$ is pure of weight 3.

The following exact sequence is part of the Mayer-Vietoris exact sequence for the discriminant square \begin{eqnarray} \label{MVseqa}\dots  \to H^2(E)\to H^3(Y)\to H^3(\tilde{Y})\to \dots  \end{eqnarray}
From this  it follows that $w_{3,k}=0$ for $k\neq 2,3$. Hence $h^3(Y)=w_{2,3}+w_{3,3}$.


\begin{lemma} Suppose $p\equiv 1 \bmod 6$. Then we have that 
\begin{eqnarray*} w_{2,3}&\geq& \frac{p^3+19p^2 +p+1-\#Y(\F_p)- 24 p^{3/2}}{p^2-2p^{3/2}+p } \mbox{ and } \\w_{2,3} &\leq &\frac{p^3+19 p^2 +p+1-\#Y(\F_p)+ 24 p^{3/2}}{p^2+2p^{3/2}+p }          \end{eqnarray*}

\end{lemma}

\begin{proof}
From the Lefschetz hyperplane theorem \cite[Theorem B22]{Dim} it follows that $h^i(Y)=0$ for $i\neq 0,2,3,4,6$, that $h^0(Y)$ and $h^6(Y)=1$. Now $H^4(Y)$ is  pure of weight 4, whereas $H^3(Y)$ has a weight 2 and a weight 3 part. 
Hence
\[ \chi(Y)=3+h^4-w_{2,3}-w_{3,3} .\]
From Proposition~\ref{prpw23} it follows that
\begin{equation}
\label{h4rel} h^4=h^4_\Sigma+1-w_{2,3} 
 \end{equation}
Substituting this and solving for $w_{3,3}$ yields
\[ w_{3,3}= h^4_\Sigma+4-2w_{2,3}-\chi(Y) .\]
Since $p\equiv 1 \bmod 6$ it follows from Proposition~\ref{prpSec2}  that  $\Frob_p$ acts as multiplication by $p^2$ on $H^4_\Sigma(Y)$. Using Proposition~\ref{prpw23} it follows that $\Frob_p$ acts as multiplication by $p^2$ on $H^4(Y)$ and as multiplication by $p$ on $\Gr^W_2H^3$.

The Lefschetz trace formula
\[ \#Y(\F_q)= \sum (-1)^i \Tr \Frob | H^i \]
 reads as
\[\# Y(\F_q)= 1+p(1-w_{2,3})-\Tr W_3H^3 +p^2 h^4+p^3\]
or, equivalently,
\[   \# Y(\F_q)-1-p+pw_{2,3}-p^2 h^4-p^3 = \Tr W_3 H^3\]
Taking absolute values yields
\[  | \# Y(\F_q)-1-p+pw_{2,3}-p^2 h^4-p^3 |= |\Tr W_3 H^3| \leq w_{3,3}p^{3/2}.\]

Combined with (\ref{h4rel}), this yields
\[  | \# Y(\F_p)-1-p+(p+p^2)w_{2,3}-p^2 h^4_\Sigma-p^2 -p^3 |\leq (h^4_\Sigma+4-\chi-2w_{2,3}) p^{3/2}.\]
Setting $h^4_\Sigma=18,\chi=-2$ (From Proposition~\ref{prpSec2} and Lemma~\ref{lemchi}) this simplifies to
\begin{eqnarray*} w_{2,3}&\geq& \frac{p^3+19p^2 +p+1-\#Y(\F_p)- 24 p^{3/2}}{p^2-2p^{3/2}+p } \mbox{ and } \\w_{2,3} &\leq &\frac{p^3+19 p^2 +p+1-\#Y(\F_p)+ 24 p^{3/2}}{p^2+2p^{3/2}+p }     .     \end{eqnarray*}
\end{proof}

\begin{proposition}
 We have that $w_{2,3}=12, w_{3,3}=0$ and  $h^{2,2}=h^4=7$.
\end{proposition}
\begin{proof}
The smallest prime congruent 1 modulo 3 is 7. Now $\#Y(\F_7)=610$. In this case the upper and lower bound coincide, and equal 12.
This implies that $w_{2,3}=12$, $h^4=7$ and $w_{3,3}=0$.
\end{proof}
\begin{remark}
The original method of Van Geemen and Werner used primes that are slightly bigger. That $p=7$ suffices follows from the fact that $w_{3,3}=0$. This latter fact implies that the upper and lower bound coincide for every prime $p\equiv 1 \bmod 3$.
\end{remark}

Collecting everything we get:
\begin{theorem} Suppose $K\subset \C$. We have
\[ H^i(Y,\Q) = \left\{ \begin{array}{cl}
 \Q(-i/2) & i=0,2,6 \\
 \Q(-1)^{12} & i=3 \\
 \Q(-2)^7 & i=4 \\
 0  & i\neq 0,2,3,4,6
 \end{array}  \right.\]
            and $\rank E(\oQ(s,t))=6$.
\end{theorem}
\begin{remark}
In the introduction we gave six independent points in $E(\oQ(s,t))$ yielding a finite index subgroup. A standard specialization argument (as in \cite{elljcst}) shows that $E(\oQ(s,t))$ is torsion free.
\end{remark}


\end{document}